\documentclass[12pt]{article}

\usepackage{lineno,hyperref}

\usepackage{epsfig,amsmath,hyperref}
\usepackage[english]{babel}
\usepackage{color}
\usepackage{amsfonts,amssymb,amsmath,amsopn,amsxtra}
\usepackage{mathrsfs}
\usepackage{enumerate}
\usepackage{bbm}
\usepackage{graphicx}
\usepackage{rotating}
\usepackage[normalem]{ulem}
\usepackage{enumerate}

\newcommand{\e}{\mathrm{e}}
\newcommand{\D}{\mathrm{d}}

\newcommand{\CC}{\mathcal{C}}
\newcommand{\N}{\mathbb{N}}
\newcommand{\R}{\mathbb{R}}
\newcommand{\Z}{\mathbb{Z}}

\newcommand{\ac}{\alpha_{\mathrm{crit}}}

\newcommand{\Hm}[1]{\leavevmode{\marginpar{\tiny%
$\hbox to 0mm{\hspace*{-0.5mm}$\leftarrow$\hss}%
\vcenter{\vrule depth 0.1mm height 0.1mm width \the\marginparwidth}%
\hbox to
0mm{\hss$\rightarrow$\hspace*{-0.5mm}}$\\\relax\raggedright #1}}}

\def\softt{{\leavevmode\setbox1=\hbox{t}%
\hbox to \wd1{t\kern-0.6ex{\char039}\hss}}}

\newtheorem{claim}{Claim}[section]
\newtheorem{theorem}[claim]{Theorem}
\newtheorem{lemma}[claim]{Lemma}

\newtheorem{proposition}[claim]{Proposition}
\newenvironment{proof}[1][Proof]{\textsl{#1.} }{\ \rule{0.4em}{0.7em}}

\title{Aharonov and Bohm \emph{vs.} Welsh eigenvalues}

\author{P. EXNER$^{1,2}$ and S. KONDEJ$^{3}$}
\date{\small $^1$Doppler Institute for Mathematical Physics and Applied Mathematics, \\ Czech Technical University in Prague, B\v{r}ehov\'{a} 7, 11519 Prague, Czechia \\ $^2$Nuclear Physics Institute CAS, 25068 \v{R}e\v{z} near Prague, Czechia \\
$^3$Institute of Physics, University of Zielona G\'ora, ul.\ Szafrana 4a, \\ 65246 Zielona G\'ora, Poland \\ \emph{e-mail: exner@ujf.cas.cz, s.kondej@if.uz.zgora.pl}}

\begin{document}

\maketitle

\noindent \textbf{Abstract.}  We consider a class of two-dimensional Schr\"odinger operator with a singular interaction of the $\delta$ type and a fixed strength $\beta$ supported by an infinite family of concentric, equidistantly spaced circles, and discuss what happens below the essential spectrum when the system is amended by an Aharonov-Bohm flux $\alpha\in [0,\frac12]$ in the center. It is shown that if $\beta\ne 0$, there is a critical value $\ac \in(0,\frac12)$ such that the  discrete spectrum has an accumulation point when $\alpha<\ac $, while for $\alpha\ge\ac $ the number of eigenvalues is at most finite, in particular, the discrete spectrum is empty for any fixed $\alpha\in (0,\frac12)$ and $|\beta|$ small enough.

\medskip


\medskip

spectrum, Aharonov-Bohm flux.

\section{Introduction}\label{s:intro}
\setcounter{equation}{0}

Schr\"odinger operators with radially periodic potentials attracted attention because they exhibit interesting spectral properties. It was noted early \cite{Hempel} that the essential spectrum threshold of such an operator coincides with that of the one-dimensional Schr\"odinger operator describing the radial motion. More surprising appeared to be the structure of the essential spectrum which may consist of interlacing intervals of dense point and absolutely continuous nature as was first illustrated using potentials of cosine shape \cite{HHHK}.

While this behavior can be observed in any dimension $\ge 2$, the two-dimensio\-nal case is of a particular interest because here these operators can also have a discrete spectrum below the threshold of the essential one. This fact was first observed in \cite{BEHKMS} and the national pride inspired the authors to refer to this spectrum as to \emph{Welsh eigenvalues}; it was soon established that that their number is infinite if the radially symmetric potential is nonzero and belongs to $L^1_\mathrm{loc}\,$ \cite{Schmidt99}. Moreover, the effect persists if such a regular potential is replaced by a periodic array of $\delta$ interactions or more general singular interactions \cite{EF07, EF08}.

The question addressed in this paper is how are the Welsh eigenvalues influenced by a local magnetic field preserving the rotational symmetry. For simplicity we will choose the simplest setting, the two-dimensional system with $\delta$ potential of a fixed strength $\beta$ supported on a concentric family of circles $\{\CC_{r_{n}}\}_{n\in \N}$ or radii $r_n=d(n+\frac12)\,,\: n=0,1,\dots\,$, with $d>0$. Without the presence of the magnetic field the corresponding Hamiltonian can be symbolically written as
$$
H_{\beta } = -\Delta +\beta \sum_{n}\delta (x-\CC_{r_{n}})\,,\quad \beta \in \R\,,
$$
which can be given meaning as a self-adjoint operator in $L^2(\R^2)$ as we will recall below. As we have said, the discrete spectrum of $H_\beta$ is infinite \cite{EF08}, which is a direct consequence of the fact that the effective potential in the s-wave component contains the term $-\frac{1}{4r^2}$ producing an infinite number of eigenvalues below $E_\beta:= \inf\sigma_\mathrm{ess}(H_\beta)$.

The magnetic interaction we add is also chosen in the simplest possible way, namely as an Aharonov-Bohm flux $\alpha$ at the origin of the coordinates, measured is suitable units, that gives rise to the magnetic field vanishing outside this point. The corresponding Hamiltonian will be denoted $H_{\alpha,\beta}$ and as we will argue, it is sufficient to consider flux values up to half of the quantum, $\alpha\in(0,\frac12)$. Since singular interactions are involved, it is maybe useful to stress that we consider an Aharonov-Bohm flux alone, without an additional point interactions at origin \emph{\`{a} la} \cite{AT, DS}. It is known that local magnetic fields generally, and Aharonov-Bohm fluxes in particular, can reduce the discrete spectrum, if combined with an effective potential that behaves like $r^{-2}$, on the borderline between short and long range, the effect can be dramatic \cite{KLO}.

We are going to show that in the present model the Aharonov-Bohm field also influences the discrete spectrum but the dependence on the flux value is more complicated. Specifically, we claim that
\begin{itemize}
\setlength{\itemsep}{-3pt}
\item there is an $\ac (\beta )=\alpha_{\mathrm{crit}} \in (0, \frac12)$ such that for $\alpha \in (0,\ac )$ the discrete spectrum of $H_{\alpha,\beta}$ is infinite  accumulating at the threshold $E_0$, while for $\alpha \in [\ac,\frac12)$ there is at most finite number of eigenvalues below $E_0$,
 \item the critical value $\ac (\beta )$ admits the following asymptotics,
 $$
 \ac (\beta )\to \frac 12 - \quad \mathrm{for }\quad \beta \to \pm \infty
 $$
 and
 $$
 \ac (\beta )\to 0 + \quad \mathrm{for }\quad \beta \to 0 \,,
 $$
\item for any fixed $\alpha \in (0, \frac12 )$ there exists $\beta_0>0$ such that for any $|\beta|\leq \beta_0$ we have
$\sigma_\D(H_{\alpha ,\beta}) = \emptyset$, and moreover,  $\sigma_\D(H_{\frac12,\beta}) = \emptyset$ holds for any $\beta \in \R$.
\end{itemize}
These properties will be demonstrated in Sections~\ref{s:discr}~and~\ref{s:finite}; before coming to that, in the next section we introduce properly the Hamiltonian and derive its elementary properties.

\section{Preliminaries} \label{s:prelim}
\setcounter{equation}{0}

We consider a magnetic flux $\phi$ perpendicular to the plane to which the particle is confined and placed at the origin of the coordinates corresponding to the vector potential
$$
A(x,y) = \frac{\phi}{2\pi }\left( -\frac{y}{r^2}, \frac{x}{r^2} \right)\,.
$$
In the rational units we use the flux quantum is $2\pi$, thus it is natural to introduce $\alpha := \frac{\phi}{2\pi}$. Given this $A$ we define the `free' Aharonov-Bohm  Hamiltonian
$$
H_\alpha := (-i\nabla - A)^2\,,\quad D(H_\alpha )= \{f\in L^2 (\R^2):\: (-i\nabla - A)^2 f \in L^2 \}\,,
$$
where the domain is sometimes dubbed magnetic Sobolev space. Since the integer part of a given $\alpha$ can be removed by a simple gauge transformation, it is sufficient to consider $\alpha\in(0,1)$ only.

The radial symmetry allows us to describe $H_\alpha$ in terms of the partial wave decomposition. To this aim we introduce unitary operator $U\,:\, L^2 (\R_+, r\mathrm{d}r)\to L^2 (\R_+) $  acting as $Uf (r)= r^{1/2} f(r)$. This  naturally leads to
$$
L^{2} (\R^2 )= \bigoplus_{l\in\Z} U^{-1} L^2 (\R_+ ) \otimes S_l\,,
$$
where $S_l$ is the $l$-th eigenspace of Laplace operator on the unit circle, and the corresponding decomposition of the Hamiltonian
$$
H_{\alpha }= \bigoplus_{l} U^{-1} H_{\alpha, l } U \otimes I_l\,,
$$
where $I_l$ is the identity operator on $S_l$ and the radial part is
\begin{align}
  H_{\alpha,l}: = & - \frac{d^2}{d^2r} +\frac{1}{r^2}c_{\alpha,l}\,,\quad c_{\alpha,l}:= -\frac{1}{4}+(l+\alpha )^2\,, \label{eq-Hamdecom} \\[.5em]
  D(H_{\alpha,l}) :=& \, \{  f\in L^2 (\R_+)
  \,:\, -f''+\frac{ c_{\alpha, l} }{r^2}f\in L^2 (\R_+)\,, \nonumber \\
  & \quad \lim _{r\to 0^+ }r^{\alpha -1/2}f(r)=0 \,, \: l=0\,, \label{eq-aux} \\
  & \quad \lim _{r\to 0^+ }r^{1-\alpha -1/2}f(r)=0 \,, \: l=-1  \}. \nonumber
 \end{align}
We recall that this operator describes a `pure' Aharonov-Bohm field without an additional singular interaction at the origin \cite{AT, DS}. This corresponds to the choice of $H_{\alpha, l}\,,\, l=0,-1$, as appropriate self-adjoint extensions of the operator $- \frac{d^2}{d^2r} +\frac{1}{r^2}c_{\alpha, l}$ restricted to $C_0^\infty(\R_+)$. For all the other values of $l$ the centrifugal term ensures the essential self-adjointness, here we choose the conditions which exclude the more singular of the two solutions at the origin, $r^{1/2}K_\alpha (\kappa r)$ and $r^{1/2}K_{1-\alpha} (\kappa r)$, respectively.

In the next step we consider the \emph{$\delta$ interaction supported by concentric circles}; we amend the system governed by $H_\alpha $ by a singular radially periodic potential supported by concentric circles $\CC_{r_n}$ of the radii $r_n =d(n+\frac12)$, $\:d>0$, the strength of which is characterized by a nonzero coupling constant $\beta\in \R$. Since the radial symmetry is preserved, the resulting Hamiltonian can be again expressed in terms of its partial-wave components,
 \begin{equation} \label{eq-Hammain}
 H_{\alpha; \beta } =  \bigoplus_{l} U^{-1} H_{\alpha;\beta , l } U \otimes I_l\,,
 \end{equation}
where 
\begin{align} \label{eq-domain1}
D(H_{\alpha;\beta,l}) := & \{ f,\in W^{2, 2} (\R^2  \setminus \cup_{n\in \N} \CC_{r_n} ): \:
f\: \mathrm{satisfies}\: (\ref{eq-aux}) \\ \label{eq-domain2} & \mathrm{and}\;
\partial_r f(r_n^+)-\partial_r f(r_n^-) = \beta f(r_n )\,,\; n\in \N \}\,;
\end{align}
it is easy to check that operator $H_{\alpha ;\beta }$ is self-adjoint.

As in \cite{EF07} it is useful to introduce a one-dimensional comparison operator which is the usual \emph{Kronig-Penney Hamiltonian} with equidistantly spaced $\delta$ interactions supported by the set $\{x_n :=d(n+\frac12)\,:\: n\in \mathbb{Z}\}$. We denote it $\mathrm{h}_\beta$, it acts as $\mathrm{h}_\beta f=-f''$ on the domain
$$
D(\mathrm{h}_\beta) =\{ f \in W^{2,2} (\R \setminus
\cup_{n\in \mathbb{Z}}\{x_n \})\,:\:
f'(x_n^+)- f'(x_n^-) = \beta f(x_n )\,,\; n\in \mathbb{Z}\}\,.
$$
Let $E_0 $ stand for the spectral threshold of $\mathrm{h}_\beta $,
\begin{equation} \label{infsigmacomp}
 E_0 := \inf \sigma  (\mathrm{h}_\beta )\,;
\end{equation}
mimicking the argument used in \cite{EF08} we can check easily that this quantity determines the essential spectrum of $H_{\alpha; \beta}$, namely
\begin{equation} \label{infsigmaess}
\sigma_{\mathrm{ess}} (H_{\alpha;\beta}) =[E_0,\infty)\,.
\end{equation}
Although it is not important for the present work, let us add that the reasoning made in \cite{EF08} remains valid if the centrifugal coefficients in
\eqref{eq-Hamdecom} replace their nonmagnetic values $c_{0,l}$, and consequently, the essential spectrum is not affected by the Aharonov-Bohm flux consisting of the absolutely continuous bands that coincide with the spectral bands of $\mathrm{h}_\beta$ and the dense point part filling the spectral gaps of $\mathrm{h}_\beta$.

Our interest here concerns the spectrum of $H_{\alpha;\beta}$ in the interval $(-\infty,E_0)$ which is discrete according to \eqref{infsigmaess}. Let us first collect its elementary properties.

\begin{proposition}
Suppose that $\beta\ne 0$, then
\vspace{-.5em}
\begin{enumerate}[(i)]
\setlength{\itemsep}{-2pt}
\item $\sharp\sigma_\mathrm{disc}(H_{0;\beta})=\infty$
\item $\sigma_\mathrm{disc}(H_{\frac12;\beta})=\emptyset$
\item $\sigma_\mathrm{disc}(H_{\alpha;\beta})=\sigma_\mathrm{disc}(H_{1-\alpha;\beta})$
\item eigenvalues of $H_{\alpha;\beta}$ are nondecreasing in $[0,\frac12]$, $\,\lambda_j(\alpha')\ge\lambda_j(\alpha)$ if $\alpha'\ge \alpha$
\end{enumerate}
\end{proposition}
\begin{proof}
Claim (i)  follows from \cite[Thm~5.1]{EF08}. Partial wave operators in the decomposition \eqref{eq-Hammain} can contribute to $\sigma_\mathrm{disc}(H_{\alpha;\beta})$ only if $c_{\alpha,l}<0$. Indeed, if $c_{\alpha,l}=0$ the spectrum of $H_{\alpha;\beta,l}$ coincides, up to multiplicity, with that of the operator $\mathrm{h}_\beta$ amended according to \eqref{eq-aux} with Dirichlet condition at $x=0$, hence \eqref{infsigmacomp} in combination with a bracketing argument \cite[Sec.~XIII.15]{RS4} shows that the discrete spectrum is empty and yields assertion (ii). Furthermore, in view of the min-max principle \cite[Sec.~XIII.1]{RS4} this verifies the above claim and shows that the discrete spectrum comes from $H_{\alpha;\beta,0}$ if $\alpha\in[0,\frac12)$ and from $H_{\alpha;\beta,-1}$ if $\alpha\in(\frac12,1)$. The third claim follows from the identity $c_{\alpha,0}=c_{1-\alpha,-1}$ valid for $\alpha\in(0,1)$, and the last one we get employing the min-max principle again.
\end{proof}

\medskip

It is therefore clear, as indicated in the introduction, that to
describe the discrete spectrum it is sufficient to limit our
attention to the values $\alpha\in(0,\frac12)$ and to consider the
operator $H_{\alpha;\beta,0}$.

\section{Properties of the discrete spectrum}  \label{s:discr}
\setcounter{equation}{0}

The previous discussion shows that the discrete spectrum comes for
$\alpha\in(0,\frac12)$ from the partial wave operator $H_{\alpha;
\beta,0}$ and the decisive quantity is the coefficient
$c_{\alpha,0}= \alpha^2-\frac14$. Let $y$ be the solution of
\begin{equation}\label{eq-aux3}
H_{\alpha;\beta,0} g=E_0 g\,,
\end{equation}
where $E_0$ is the threshold value \eqref{infsigmacomp}. We are
going to employ the oscillation theory; following its general
strategy we introduce the Pr\"ufer variables $(\rho, \theta )$ as
follows
$$
\left(
  \begin{array}{c}
    y \\ y'
  \end{array}
\right)= \rho
\left(
  \begin{array}{c}
    \cos \theta \\ \sin \theta
  \end{array}
\right)\,.
$$
As it is usually the case with singular potentials \cite{EF08}, we
can rephrase the discrete spectrum analysis as investigation of the
asymptotic behavior of the function $r\mapsto \theta(r)$; for the
reader convenience the needed facts from the oscillation theory are
collected in Sec.~\ref{s:tools} below.

To formulate the first main result we denote by $u$ the $d
$-periodic real-valued solution of the one-dimensional comparison
problem,
\begin{equation}\label{eq-hbeta}
\mathrm{h}_\beta u=E_0 u\,.
\end{equation}

Then we can make the following claim:
\begin{theorem} \label{th-main}
Suppose that $\alpha\in(0,\frac12)$ and put
$$
c_{\mathrm{crit}}:= - \frac{1}{4} \left(\frac{1}{d}\int_{0}^d
\frac{1}{u^2}\,\D x\right)^{-1} \left(\frac{1}{d} \int_{0}^d u^2\,\D
x\right)^{-1}
$$
then $E_0$ is an accumulation point of
$\sigma_\mathrm{disc}(H_{\alpha; \beta,0})$ provided
$\frac{c_{\alpha,0}} {c_{\mathrm{crit}}} >1$, while for
$\frac{c_{\alpha,0}} {c_{\mathrm{crit}}} \leq 1$ the operator has at
most finite number of eigenvalues below $E_0$ with the multiplicity
taken into account.
\end{theorem}
\begin{proof}
The asymptotic properties of the function $\theta$ can be found in a
way similar to that used in \cite{Schmidt2000}. Let $u,\,v$ be
linearly independent real-valued solutions of equation
(\ref{eq-hbeta}), where $u$ is the mentioned $d$-periodic function
involved in the definition of $c_{\mathrm{crit}}$, chosen in such a
way that the Wronskian $W[u,v]=1$. Furthermore, we introduce the
generalized Pr\"ufer variables
\begin{equation}\label{eq-anPR}
\left(
  \begin{array}{c}
    y \\
    y' \\
  \end{array}
\right) = \left(
            \begin{array}{cc}
              u & v \\
              u' & v' \\
            \end{array}
          \right) a \left(
                      \begin{array}{c}
                        \sin \gamma  \\
                      -\cos \gamma \end{array}
                    \right)\,,
\end{equation}
where $a$ is a smooth positive function and $\gamma $ is continuous
in view of \cite[Lemma~3.4]{EF08}. On the other hand, according to
\cite[Prop.~1]{Schmidt2000} the functions $\gamma (\cdot)$ and
$\theta (\cdot)$ have the same asymptotics up to the constant.
Consequently, it is sufficient to investigate the asymptotics of
$\gamma (\cdot)$ which we will do using the expression
$$
\gamma ' = \frac{c_{\alpha,0}}{r^2}\left( u\sin \gamma -v \cos
\gamma \right)^{2}= c_{\alpha,0}\,u^2\,\cos^2 \gamma\,
\left(\frac{1}{r}\tan \gamma -\frac{v}{ru}\right)^2\,,
$$
which can be obtained from (\ref{eq-anPR}) by a direct computation
using \eqref{eq-aux3} and the Wronskian properties of the functions
$u,v$. In the next step we employ the Kepler transformation
$$
\tan \phi = \frac{1}{r}\, \tan \gamma -\frac{1}{r}\frac{v}{u}\,,
$$
which yields
\begin{equation}\label{eq-aux4}
  \phi' = \frac{1}{r} \left( -\sin \phi \cos \phi +\mathcal{B}(r) \sin ^2 \phi  +
  \mathcal{A}(r) \cos^2 \phi \right)\,,
\end{equation}
where $\mathcal{A}$ and $\mathcal{B}$ are the $d$-periodic functions
defined by
\begin{equation}\label{eq-defAB}
\mathcal{B}(r):=c_{\alpha,0} u(r)^2 \quad\; \mathrm{and } \quad\;
\mathcal{A}(r):= - \frac{1}{u(r)^2}\,.
\end{equation}
The Kepler transformation preserves the asymptotics, i.e. $\gamma
(r)= \phi (r)+\mathcal{O}(1)$ holds as $r\to \infty$, thus we may
inspect the asymptotics of $\phi (\cdot)$. This can be done in the
same way as for regular period potentials. Specifically, we define
\begin{equation}\label{eq-phidash}
  \overline{\phi} (r):=\frac{1}{d }\int_{r}^{r+d}\phi(\xi)\,\D\xi \,, \quad r>R_0\,,
\end{equation}
for some $R_0>0$. Proposition 2 of \cite{Schmidt2000} allows us to
conclude that $\overline{\phi}(r)=\phi(r)+o(1)$ and
\begin{equation}\label{eq-aux5}
  \overline{\phi}'(r) = \frac{1}{r} \left( -\sin \overline{\phi} \cos \overline{\phi} + B \sin ^2 \overline{\phi}  +
 A \cos^2 \overline{\phi} \right)+\mathcal{O}(r^{-2})\,,
\end{equation}
where
$$
A:=\frac{1}{d}\int_{0}^d \mathcal{A}(r)\,\D r \quad\; \mathrm{and }
\quad\; B:=\frac{1}{d}\int_{0}^d \mathcal{B}(r)\,\D r\,.
$$
Now we apply Proposition 3 of \cite{Schmidt2000} which states that
$\overline{\phi}$ is bounded provided $4AB<1$ and unbounded if
$4AB>1$. Combining this fact with the observation that
$$
4AB = -4c_{\alpha, 0} \left(\frac{1}{d }\int_{0}^d \frac{1}{u^2}\,\D
r\right) \left(\frac{1}{d }\int_{0}^d u^2\,\D r\right) =
\frac{c_{\alpha,0}}{c_{\mathrm{crit}}}
$$
we come to the claim of  theorem for any $c_{\alpha , 0}$ apart from the case
 $c_{\alpha , 0} = c_{\mathrm{crit}}$.  To complete the proof we note that
 $$
 \lim_{r\to 0 }\,r(\log r )^2 \left(  \overline{\phi}'(r) - \frac{1}{r} \left( -\sin \overline{\phi} \cos \overline{\phi} + B \sin ^2 \overline{\phi}  +
 A \cos^2 \overline{\phi} \right)\right) =0 \,,
 $$
 cf.~(\ref{eq-aux5}). Applying now Proposition~4~of~\cite{Schmidt2000} we conclude that, if $4AB =1$ then $ \overline{\phi}$ is globally bounded. This equivalently means that for $c_{\alpha , 0} = c_{\mathrm{crit}}$ at most finite number of discrete spectrum below $E_0$ can exist.
\end{proof}

\bigskip

This allows us to prove the following claim.

\begin{theorem}
There exists an $\ac (\beta )=\ac \in (0,\frac12)$ such that for $\alpha \in
(0, \ac)$ the operator $H_{\alpha, \beta }$ has infinitely
many eigenvalues accumulating at the threshold $E_0$, the
multiplicity taken into account, while for $\alpha \in
[\ac ,\frac12)$ the cardinality of discrete spectrum is finite.
\end{theorem}
\begin{proof}
The function $\alpha \mapsto c_{\alpha , 0} =\alpha^2 -\frac14$
is increasing in $(0,\frac12)$. Thus it suffices to show that
$c_{\mathrm{crit}}\in(-\frac14,0)$ which is an easy consequence of
Schwartz inequality,
$$
c_{\mathrm{crit}}:= - \frac{1}{4} \left(\frac{1}{d}\int_{0}^d
\frac{1}{u^2}\,\D x\right)^{-1} \left(\frac{1}{d} \int_{0}^d u^2\,\D
x\right)^{-1} > - \frac{1}{4} \left(\frac{1}{d }\int_{0}^d \D
x\right)^{-2} = - \frac{1}{4}\,;
$$
note that the inequality is sharp because the function $u$ is
nonconstant. The claim then follows from Theorem~\ref{th-main} if we
set $\ac  := \sqrt{c_{\mathrm{crit}}+\frac14}$.
\end{proof}

\bigskip

Moreover, in our present case  the critical value can be computed
explicitly because we know the function $u$ which is equal to
\begin{equation}\label{eq-gefunc}
u(x)  = \left\{
  \begin{array}{ll}
 \e^{-\kappa_0 (x-d/2) }+\e^{\kappa_0 d}\e^{\kappa_0 (x-d/2)} & \;\hbox{for }  0<x <\frac{d}{2}\,,
\\[.5em]
   \e^{\kappa_0 d} \e^{-\kappa_0 (x-d/2)}+    \e^{\kappa_0 (x-d/2)} & \;\hbox{for }  \frac{d}{2}<x<d
  \end{array}
\right.
\end{equation} 
cf.~\cite[Sec.~III.2.3]{AGHH}, where $i\kappa_0 =  k_0$,
$k_0^2=E_0$. Note that the function is obviously real-valued if
$\beta<0$ so that $E_0<0$ and $\kappa_0>0$, in the opposite case with
$\beta>0$ we have $E_0>0$ and $\kappa_0$ is purely imaginary,
nevertheless $u$ is a multiple of a real-valued function again. A
straightforward calculations then yields
$$
D_1: = \frac{1}{d }\int_{0}^d u^2\,\D x =\frac{2}{d}\,\e^{\kappa_0
d} \left(\frac{1}{2\kappa_0 } \left(\e^{\kappa_0 d }- \e^{-\kappa_0
d }\right)+d\right)
$$
and
$$
D_2: = \frac{1}{d }\int_{0}^d \frac{1}{u^{2}}\,\D x
=\frac{1}{d\kappa_0} \,\e^{-\kappa_0 d}\left( \frac{1}{2} -
\frac{1}{1+\e^{\kappa_0 d }}\right)\,.
$$
Using this notation we have
$$
c_{\mathrm{crit}}=-\frac{1}{4} \frac{1}{D_1 D_2}\,.
$$

These expressions allow us, in particular, to find the behavior of
the critical flux values in the asymptotic regimes. In the
\emph{weak coupling constant case}, $\beta \to 0 $, we have
$\kappa_0 \to 0 $ and the quantities $D_1$ and $D_2$ have the
following limits
$$
D_1 \to 4 \,,\quad D_2\to \frac{1}{4}\;\quad \mathrm{as } \quad
\beta \to 0 \,.
$$
This implies $c_{\mathrm{crit}}\to -\frac{1}{4}$, and therefore
\begin{equation} \label{weak}
\ac (\beta)\to 0+\;\quad \mathrm{as } \quad \beta \to 0
\end{equation}
which is certainly not surprising in view of the fact that the
discrete spectrum is empty for $\beta=0$.

In the \emph{strong coupling constant case} one has to take the sign
of $\beta$ into account as the spectral condition takes a different
form,
$$
\coth \Big(\frac12\kappa d\Big)=\frac{2\kappa}{|\beta |} \qquad
\mathrm{and} \qquad \cot\Big(\frac12 kd\Big) = \frac{2k}\beta
$$
for $\mp\beta\to\infty$, and $E_0(\beta)$ tends to $-\infty$ and
$\big(\frac{\pi}{d}\big)^2$, respectively. In both cases, however,
$c_{\mathrm{crit}}$ tends to zero, exponentially fast for the
attractive $\delta$ interactions when
$$
c_{\mathrm{crit}} \approx -\frac{d ^2}{8}\, \e^{-|\beta | d/2}\,.
$$
Furthermore, this yields
\begin{equation} \label{weak}
\ac (\beta)\to \frac12- \;\quad \mathrm{as } \quad \beta \to
\pm\infty\,.
\end{equation}

\section{Nonexistence of the discrete spectrum for weak $\delta$ interactions}\label{s:finite}
\setcounter{equation}{0}

The above results tell us nothing about the spectrum of $H_{\alpha ; \beta }$ for $\alpha\in[\ac,\frac12)$, in particular, we do not know whether the operator may have some eigenvalues. Our aim now is to show that for a fixed $\alpha$, with the exception of the nonmagnetic and half-of-the-quantum cases, we have
$$
\sigma_{\mathrm{disc}} (H_{\alpha ;\beta }) = \emptyset\,.
$$
provided the involved $\delta$ interaction is sufficiently weak. Using a modified version of the Hardy inequality, we are going to prove the following claim:
\begin{theorem}\label{th-disc}
  Given $\alpha \in (0, \frac12 )$  there exists a $\beta _0 >0$ such that
  for any $|\beta | <\beta_0 $ the operator $H_{\alpha; \beta }$ has no discrete spectrum.
\end{theorem}
\begin{proof}
To show that the discrete spectrum is void it suffices to investigate the `lowest' partial-wave component $H_{\alpha; \beta , 0}$. Consider the quadratic form associated with the `shifted' operator $H_{\alpha; \beta , 0}-E_0$,
\begin{align}\label{eq1}
 & q_{\alpha; \beta , 0} [f] := \\[.3em] & \int_{0}^\infty |f(r)'|^2\,\mathrm{d} r + c_{\alpha , 0 }\int_{0}^\infty \frac{1}{r^2}|f(r)|^2 \,\mathrm{d} r + \beta \sum_{n}\int_{\mathcal{C }_{r_n }} |f (r)|^2\,\mathrm{d} \mu_{{ \mathcal{C} _{r_n }}} - E_0 \|f\|^2\,, \nonumber
\end{align}
where $ \mu_{ \mathcal{C} _{r_n } }$ defines the arc length measure on $ \mathcal{C} _{r_n } $ and  $c_{\alpha , 0}\in (-1/4,0)$, moreover, $f\in D(H_{\alpha; \beta , 0})$, i.e. that it satisfies the boundary conditions given by (\ref{eq-domain1})  and   (\ref{eq-domain2}). Without loss of generality we may assume that $f$ is a  real function. As in the previous discussion  $u$ stands for the periodic  function defining the `lowest'generalized eigenfunction of $H_{\alpha; \beta , 0}$. We may assume that $u$ is positive, then from the explicit expression (\ref{eq-gefunc}) we see that for a fixed $\beta_1 >0$ there exists a $C_{\mathrm{min}}>0$ such that $u \geq C_{\mathrm{min}}$ holds for any $|\beta |\leq \beta_1$. Furthermore, we put $\chi =\frac{f}{u}$; one can easily check that $\chi \in H_0^{2,2} (\R_+)$.  Integrating by parts and using the boundary conditions (\ref{eq-domain1}) and (\ref{eq-domain2}) we get
  $$
  q_{\alpha ;\beta ,0}[u\chi ] = -\int_{0}^\infty  u  \chi (u  \chi)'' \,\mathrm{d}r+  c_{\alpha, 0}\int_{0}^\infty u^2\frac{\chi ^2}{r^2} \,\mathrm{d}r  - E_0 \| u  \chi \|^2\,.
  $$
After expanding the second derivative and using the equation that $u$ as a generalized eigenfunction satisfies we get
  \begin{eqnarray}
     q_{\alpha ;\beta ,0}[u\chi ] &=& \int_{0}^\infty u ^2 \left(  -\chi  \chi '' +  \frac{ c_{\alpha ,0} }{r^2}\chi ^2 \right)\,\mathrm{d}r
  -  \int_{0}^\infty (u^2 )'  \chi \chi' \,\mathrm{d}r\, \nonumber
   \\ \label{eq-auxIII}
   &=& \int_{0}^\infty u ^2   (\chi')^2 \,\mathrm{d}r +  c_{\alpha ,0} \int_{0}^\infty u^2 \frac{ \chi ^2 }{r^2} \,\mathrm{d}r \,,
  \end{eqnarray}
where in the second step we performed integration by parts in the last expression of the first line with the boundary term vanishing due to \eqref{eq-aux}. The following lemma will be useful in the further discussion.
\begin{lemma} \label{l:est1}
  We have
  \begin{equation}\label{eq-HardyII}
     q_{\alpha ;\beta ,0}[u\chi ]>
    \alpha^2 \int_{0}^\infty u^2 \frac{ \chi ^2 }{r^2}\,\mathrm{d}r  - \frac{1}{2} \int_{0}^\infty (u^2)' \frac{ \chi ^2 }{r} \,\mathrm{d}r\,.
  \end{equation}
\end{lemma}
  \begin{proof}
To prove the claim we start from the expression
  \begin{eqnarray}
    \int_{0}^\infty u^2  (( r^{-1/2 }\chi )' )^2 \,r \mathrm{d}r=
    \int_{0}^\infty u^2  \left( -\frac{1}{2} r^{-3/2 }\chi  + r^{-1/2 }\chi ' \right)^2 r \mathrm{d}r \nonumber
      \\ \label{eq-auxa}
     = \int_{0}^\infty u^2  \left(  \frac{1}{4r^2}\chi ^2   -\frac{1}{r} \chi \chi ' + (\chi ')^2 \right) \mathrm{d}r\,.
  \end{eqnarray}
On the other hand, the second term in (\ref{eq-auxa}) can be rewritten as
$$
    -\int_{0}^\infty u^2 \frac{\chi \chi'}{r}\,\mathrm{d}r  =-\frac{1}{2} \int_{0}^\infty u^2  \frac{(\chi ^2)'}{ r}\mathrm{d}r
    = \frac{1}{2}  \int_{0}^\infty \left(  \frac{(u^2)' }{r} - \frac{u^2}{ r^2}\right) \chi ^2 \,\mathrm{d}r\,,
$$
where we have again employed integration by parts in combination with \eqref{eq-aux}; inserting this to (\ref{eq-auxa}) we get
  $$
  \int_{0}^\infty u^2  (( r^{-1/2 }\chi )' )^2 \,r \mathrm{d}r =
    \int_{0}^\infty u^2  \left( (\chi')^2-\frac{\chi^2}{4r^2} \right)\,\mathrm{d}r +
    \frac{1}{2}  \int_{0}^\infty  \frac{(u^2)' }{r} \chi^2 \,\mathrm{d}r\,.
  $$
Since $\int_{0}^\infty u^2  (( r^{-1/2 }\chi )' )^2 r  \mathrm{d}r>0$, taking into account expression (\ref{eq-auxIII}) and using $c_{\alpha, 0}= \alpha^2 -\frac{1}{4}$ we   obtain the claim of lemma.
  \end{proof}

\medskip

With a further purpose in mind we introduce a symbol for the second term at the right-hand side of (\ref{eq-HardyII}),
$$
 \tilde{q}[\chi ] := -\frac{1}{2}  \int_{0}^\infty (u^2)'\, \frac{\chi^2 }{r} \,\mathrm{d}r\,.
$$
Our next aim it to show that $\tilde{q}[\cdot  ]$ is small with respect to $q_{\alpha; \beta , 0}[\cdot ]$. This the contents of the following lemma.
\begin{lemma} \label{l:est2}
We have
\begin{equation}\label{eq-estimqtilde}
|\tilde{q}[\chi ] | \leq \eta (\beta ) \int_{0}^{\infty } (\chi')^2 \,\mathrm{d}r\,,
\end{equation}
where the function $\eta (\cdot )$ behave asymptotically as
 $$
 \eta (\beta )=\mathcal{O}(\kappa _0 (\beta ))
 $$
for $\beta $ small. Here $\kappa _0 = \kappa _0 (\beta )$ is the quantity introduced in \eqref{eq-gefunc} and  the expression on the right-hand side does not depend on $\chi$.
\end{lemma}
\begin{proof}
Note first that an integration by parts in combination with conditions \eqref{eq-aux} yields
\begin{eqnarray*}
  \tilde{q}[\chi ]  = \frac{1}{2}  \int_{0}^\infty u^2 \left( \frac{\chi^2 }{r}\right)' \mathrm{d}r=
  \frac{1}{2}  \int_{0}^\infty \left(u^2(r)-u^{2}(0)\right) \left( \frac{\chi^2 }{r}\right)' \mathrm{d}r=\\
  =\frac{1}{2}  \int_{0}^\infty \left(u^2(r)-u^{2}(0)\right) \left( \frac{2\chi \chi ' }{r} - \frac{\chi^2}{r^2}\right)\mathrm{d}r\,.
\end{eqnarray*}
On the other hand, from the explicit expression (\ref{eq-gefunc}) we get easily
$$
\left| u^2(r)-u^{2}(0)\right| = \mathcal{O}(\kappa_0 (\beta ))
$$
as $\beta \to 0 $ where the right-hand side does not depend on $r$ since the function $u$ is periodic, and naturally neither on $\chi$. Consequently,
\begin{eqnarray} \label{eq-auxIV}
  |\tilde{q}[\chi ] | \leq   \eta_1 (\beta )\left(\int_{0}^\infty  \frac{|2\chi \chi '| }{r} \,\mathrm{d}r
  +  \int_{0}^\infty  \frac{(\chi )^2 }{r^2} \,\mathrm{d}r \right) \,,
\end{eqnarray}
where $\eta_1 (\beta )$ behaves asymptotically as $\eta_1 (\beta )= \mathcal{O}(\kappa_0 (\beta ))$.
Our next aim is to estimate the first integral on the right-hand side of (\ref{eq-auxIV}),
 $$
   \tilde{q}_1[\chi ] := 2 \int_{0}^\infty  \frac{|\chi \chi '| }{r} \,\mathrm{d}r\,.
 $$
Applying the Schwartz inequality together with the classical Hardy inequality,
 $$
\int_{0}^\infty   (\chi ')^2 \,\mathrm{d}r
 > \frac{1}{4} \int_{0}^\infty  \frac{  \chi^2 }{r^2} \,\mathrm{d}r\,,
 $$
 one obtains
 \begin{eqnarray*}
  \lefteqn{\tilde{q}_1[\chi ] \leq 2 \left(\int_{0}^\infty  \frac{(\chi )^2 }{r^2} \,\mathrm{d}r \right)^{1/2} \left(\int_{0}^\infty  (\chi ' )^2  \,\mathrm{d}r \right)^{1/2}} \\ &&
    \leq \int_{0}^\infty  \frac{(\chi )^2 }{r^2} \,\mathrm{d}r + \int_{0}^\infty  (\chi ' )^2 \,\mathrm{d}r < 5 \int_{0}^\infty  (\chi ' )^2 \,\mathrm{d}r\,.
 \end{eqnarray*}
Applying the Hardy inequality again to (\ref{eq-auxIV}) and combining this with the above result we get
 $$
 |\tilde{q}[\chi ] | \leq   \eta (\beta ) \int_{0}^\infty  (\chi ' )^2 \,\mathrm{d}r\,,
 $$
where $\eta (\beta ):= 6 \eta_1 (\beta )$. This completes the proof of lemma.
\end{proof}

\medskip

\noindent \emph{ Proof of Theorem~\ref{th-disc}, continued:} As we noted above, for any $\beta $ satisfying  $|\beta |\leq \beta_1 $ we have $\min_{r\ge 0} u(r)\ge C_\mathrm{min}$. Then the above lemma tells us that
 \begin{equation}\label{eq-gtilde}
 |\tilde{q}[\chi ] | \leq   \eta (\beta ) \frac{1}{C_\mathrm{min}}\int_{0}^\infty u^2 (\chi ' )^2 \,\mathrm{d}r\,,
 \end{equation}
Since by Lemma~\ref{l:est1} we have
\begin{equation}\label{eq-auxVI}
     q_{\alpha ;\beta ,0}[u\chi ] >
    \alpha^2 \int_{0}^\infty u^2 \frac{ \chi ^2 }{r^2} \mathrm{d}r+ \tilde{q}[\chi ]\,,
\end{equation}
relation (\ref{eq-gtilde}) yields
\begin{equation}\label{eq-auxVII}
\left| \tilde{q}[\chi ]\right|\leq \tilde{\eta} (\beta )
\int_{0}^\infty u ^2   (\chi')^2 \mathrm{d}r \quad \text{with}\quad \tilde{\eta} (\beta )
= \mathcal{O}(\kappa_0 (\beta ))\,.
\end{equation}
Combining relations (\ref{eq-auxVII}) and (\ref{eq-auxVI}) we get
$$
\left(1+\tilde{\eta }(\beta )\right) \left( \int_{0}^\infty\! u ^2   (\chi')^2 \,\mathrm{d}r +  c_{\alpha ,0} \int_{0}^\infty\! u^2 \frac{ \chi ^2 }{r^2}\,\mathrm{d}r \right)
> \left( \alpha^2 +  c_{\alpha , 0} \tilde{\eta }(\beta ) \right)\int_{0}^\infty\! u^2 \frac{ \chi ^2 }{r^2} \,\mathrm{d}r
$$
which implies
$$ 
   q_{\alpha ;\beta ,0}[u\chi ]  >
   \frac{\alpha^2 +  c_{\alpha , 0} \tilde{\eta }(\beta )}
   {1+\tilde{\eta }(\beta )}\: \int_{0}^\infty u^2 \frac{ \chi ^2 }{r^2} \,\mathrm{d}r\,.
$$ 
By Lemma~\ref{l:est2} there is a $\beta_0\in (0,\beta_1)$ such that for any $\beta$ satsisfying  $|\beta|\leq \beta_0$ the pre-integral factor in the last formula is positive which means that we have
$$
(H_{\alpha; \beta, 0 }f,f )- E_0 \|f\|^2 >0\,
$$
for any real function $f\in D(H_{\alpha; \beta, 0 } )$. The same holds \emph{mutatis mutandis} for the full Hamiltonian $H_{\alpha; \beta}$ which completes the proof.
\end{proof}

\section{Oscillation theory tools}\label{s:tools}
\setcounter{equation}{0}

To make the paper self-contained, we collect in this section the
needed results of oscillation theory for singular potentials derived
in \cite{EF08}. Note that they extend the theory of Wronskian zeros
for regular potentials developed in \cite{GST}, related results can
also be found in \cite{Weidmann}.

Consider points interaction localized at $x_n \in (l_-, l_+)$, where
$n\in M\in\N$. Moreover, assume that $q\in L^1_{\mathrm{loc} } (l_-,
l_+)$ and combine the singular and regular potential in the operator
on $L^2(l_-, l_+)$ acting as
$$
Tu(x)= -u''(x)+q(x)u(x)\,,
$$
with the domain
 \begin{align*}
D(T)  :=& \{ f, f,\in AC_{\mathrm{loc}} (l_-, l_+) \setminus
\left(\cup_{n\in M}\{x_n\}\right) \,: \, \\ & Tu\in L^2 (l_-,
l_+)\,,\,\, \partial_r f(x_n^+)-\partial_r f(x_n^-) = \beta f(x_n
)\,,\; n\in M \}\,.
 \end{align*}
In general, the operator $T$ is symmetric and we denote by $H$ its
self-adjoint extension satisfying either one of the following
conditions
\begin{itemize}
\item $T$ is limit point in at least one endpoint $l_\pm$
\item  $H $ is defined by separated boundary conditions at the
endpoints
\end{itemize}
Suppose that there exist $\psi_\pm $ that satisfy the boundary
conditions defining $H$ at $l_\pm$ and $T\psi_\pm =E\psi_\pm$.
Furthermore, let $W_0 (u_1, u_2)$ stand for the number of zeros of
the Wronskian $W (u_1, u_2)= u_1u_2'-u_1'u_2$ in $(l_-, l_+)$ and
denote $N_0 (E_1, E_2) := \mathrm{dim}\, \mathrm{Ran} P_{(E_1,
E_2)}$, where $E_1 < E_2$ and $P_{(E_1, E_2)}$ is the corresponding
spectral measure of $H$. Then we have \cite{EF08}
\begin{equation}\label{eq-wron}
  W_0 (\psi_- (E_1),\psi_+ (E_2) ) = N_0 (E_1, E_2)\,.
\end{equation}
In particular, the above equivalence allows us to estimate the
cardinality of the discrete spectrum below the essential spectrum
threshold $E_0$. Indeed, suppose $E<E_0$. Then, in the same way as
for regular potentials, there exist $u=\psi_\pm (E)$ with the
corresponding Pr\"ufer angle $\theta$ bounded for $E$ large
negative. Expressing the Wronskian in the terms of the Pr\"ufer
variables $W[\psi_-(E),\psi_+(E_0) ]= \rho (x)\rho_0(x) \sin
(\theta_0 (x)-\theta (x))$ we come to the conclusion that \emph{the
number of discrete spectrum points of $H$ below $E_0$ is finite
\emph{iff} $\theta_0 (\cdot)$ is bounded.}

\section{Concluding remarks}\label{s: concl}
\setcounter{equation}{0}

The main aim of this letter is to show that the influence of a local magnetic field on the Welsh eigenvalues depends nontrivially on the magnetic flux. In order to make the exposition simple we focused on the simple setting with radial $\delta$ potentials and an Aharonov-Bohm field, however, we are convinced that the conclusions extend to other potentials magnetic field profiles, as long as the radial symmetry and periodicity are preserved. This could be a subject of further investigation, as well as the remaining spectral properties of the present simple model such the eigenvalue accumulation for $\alpha \in (0,\ac )$ or (non)existence of eigenvalues for $\alpha \in [\ac, \frac12 )$ and an arbitrary $\beta\ne 0$. It would be also interesting to revisit from the present point of view situations in which the radially periodic interaction is of a purely magnetic type with zero total flux \cite{H}.

\subsection*{Acknowledgements}

The work was supported by the project 17-01706S of the the Czech Science Foundation (GA\v{C}R) and the project DEC-2013/11/B/ST1/03067 of the Polish National Science Centre (NCN).


\begin{thebibliography}{99}
\setlength{\itemsep}{3pt}
\bibitem{AT}   
R.~Adami, A.~Teta: On the Aharonov--Bohm Hamiltonian, \emph{Lett.
Math. Phys.} \textbf{43} (1998), 43--54.
\bibitem{AGHH}
S.~Albeverio, F.~Gesztesy, R.~H\o egh-Krohn, H.~Holden: {\em
Solvable Models in Quantum Mechanics}, 2nd edition, AMS, Providence,
R.I., 2005.
 \bibitem{AS}
M.~Abramowitz and I.~Stegun: {\em Handbook of Mathematical
Functions}, National Bureau of Standards, 1972.
\bibitem{BEKS}
J.F.~Brasche, P.~Exner, Yu.A.~Kuperin, P.~\v{S}eba:
Schr\"odinger operators with singular interactions, \emph{J.
Math. Anal. Appl.} {\bf 184} (1994), 112--139.
\bibitem{BEHKMS}
B.M.~Brown, M.S.P.~Eastham , A.M.~Hinz, T.~Kriecherbauer,
D.K.R.~McCornack, K.~Schmidt: Welsh eigenvalues of radially periodic
Schr\"odinger operators, \emph{J. Math. Anal. Appl.} {\bf 225}
(1998), 347--357.
\bibitem{DS}
L.~D\c{a}browski, P.~\v{S}{\softt}ov\'{i}\v{c}ek: Aharonov-Bohm effect
with $\delta $ type interaction, \emph{J. Math. Phys.} {\bf 39}
(1998), 47--72.
\bibitem{EF07}
P.~Exner, F.~Fraas: On the dense point and absolutely continuous
spectrum for Hamiltonians with concentric $\delta$ shells,
\emph{Lett. Math. Phys.} {\bf 82} (2007), 25--37.
\bibitem{EF08}
P.~Exner, M.~Fraas: Interlaced densed point and absolutely
continuous spectra for Hamiltonians with concetric-shell singular
interaction, \emph{Proceedings of the QMath10 Conference (Moeciu
2007)}, World Scientific, Singapore 2008, 48--65.
\bibitem{GST}
F.~Gesztesy, B.~Simon, G.~Teschl: Zeros of the wronskian and renormalization oscillation theory,
\emph{\em Americal Journal of Mathematics}
{\bf 118}(3) (1996), 571-594.
\bibitem{Hempel}
R.~Hempel, A.M.~Hinz, H.~Kalf: On the essential spectrum of
Schr\"odinger operators with spherically symmetric potentials,
\emph{Math. Ann.} {\bf 277} (1987), 197--208.
\bibitem{HHHK}
R.~Hempel, I.~Herbst, Ira, A.M.~Hinz, H.~Kalf: Intervals of dense
point spectrum for spherically symmetric Schr?dinger operators of
the type $-\Delta+\cos|x|$, \emph{J. Lond. Math. Soc., II. Ser.}
\textbf{43} (1991), 295--304.
\bibitem{H}
G.~Hoever: On the spectrum of two-dimensional Schr\"odinger operators
with spherically symmetric, radially periodic magnetic fields.
\emph{Commun. Math. Phys.} {\bf 189} (1990), 879--890.
 \bibitem{Kato}
T.~Kato: \emph{Pertubation Theory for Linear Operators},
Springer-Verlag, Berlin-Heidelberg-New York 1995.
\bibitem{KLO}
D.~Krej\v{c}i\v{r}\'{\i}k, V.~ Lotoreichik, T.~Ourmi\`eres-Bonafos:
Spectral transitions for Aharonov-Bohm Laplacians on conical layers, \emph{Proc. Roy. Soc. Edinburgh} \textbf{A}, to appear; \texttt{arXiv:1607.02454}
\bibitem{RS4}
M.~Reed, B.~Simon: \emph{Methods of Modern Mathematical Physics IV:
Analysis of Operators}, Academic Press, New York, London 1978
\bibitem{Schmidt99}
K.M.~Schmidt: Oscillation of the perturbed Hill equation and the
lower spectrum of radially periodic Schr\"odinger operators in the
plane, \emph{Proc. Americ. Math. Soc.} {\bf  127} (1999),
2367--2374.
\bibitem{Schmidt2000}
K.M.~Schmidt: Critical coupling constants and eigenvalue asymptotics
of perturbed periodic Sturm-Liouville operators, \emph{Commun. Math.
Phys.} {\bf 211} (2000), 645--685.
 \bibitem{Weidmann}
J.~Weidmann: \emph{Spectral Theory of Ordinary Differential
Operators}, Springer, Berlin 1987.














 \end{thebibliography}
\end{document}